\documentclass[11pt]{article}
\usepackage{amsthm,amsmath,amssymb}
\usepackage[latin1]{inputenc}
\newtheorem{theorem}{Theorem}
\newtheorem{definition}[theorem]{Definition}
\newtheorem{lemma}[theorem]{Lemma}

\newtheorem{corollary}[theorem]{Corollary}
\newtheorem{example}[theorem]{Example}

\newtheorem{remark}[theorem]{Remark}

\makeatletter

\@addtoreset{equation}{section}
\makeatother

\title{Second Variation of $F$-Einstein-Hilbert Functional}

\author{Ahmed Mohammed Cherif\footnote{University Mustapha Stambouli Mascara, Faculty of Exact Sciences, Mascara 29000, Algeria. Email: a.mohammedcherif@univ-mascara.dz}}
\date{}
\begin{document}
\maketitle

\begin{abstract}
This article describes a formula for second variation of generalized Einstein-Hilbert functional on Riemannian manifolds.
This work extends the definition of stable Einstein manifolds, and we present some properties.
\end{abstract}

\begin{flushleft}
Subjclass: 53C25, 83C05.
\end{flushleft}

\maketitle

\section{Introduction}
 The Einstein-Hilbert functional $\mathcal{E}$ associates to each Riemannian metric $g$ the integral of its scalar curvature $S$, that is
\begin{equation}\label{eq0}
\mathcal{E}:\mathcal{M}\longrightarrow\mathbb{R},\quad g\longmapsto \mathcal{E}(g)=\int_M Sv^g,
\end{equation}
where $\mathcal{M}$ is the set of smooth Riemannian metrics on $M$, and $v^g$ the volume form with respect
to $g$. It is the action functional that defines the dynamics of gravity in general relativity \cite{berger,Besse,calabi1,calabi,hilbert,Richard}.\\
 One of the simplest modifications to general relativity is the $F(S)$ gravity in which the Lagrangian density $F$ is an arbitrary smooth function of the scalar curvature $S$ of a Riemannian manifold $(M,g)$. When $F(s)=s$, gives the classical Einstein-Hilbert functional, therefore the Einstein gravity, corresponds to $F(S)=S$.
 The Euler-Lagrange equation of the generalized Einstein-Hilbert functional (it is known by Einstein-Hilbert functional in $f(R)$ gravity, or briefly $F$-Einstein-Hilbert functional) with respect to $g$ is proved by A. D. Felice, S. Tsujikawa in \cite{Felice}, and T. P. Sotiriou, V. Faraoni in \cite{Sotiriou}.\\
 The second variation of Einstein-Hilbert functional at Einstein metrics was considered in \cite{Koiso}.
 In \cite{Klaus}, K. Kr\"{o}ncke  study the second variation of the Einstein-Hilbert functional on Einstein metrics,
 he find some conditions for stability of Einstein manifolds with respect to the Einstein-Hilbert functional, i.e., that the second variation of the Einstein-Hilbert functional at the metric is nonpositive in the direction of transverse-traceless tensors. Stability properties of compact Riemannian Einstein manifold play a role in mathematical general relativity \cite{AM}, and in geometric analysis to understand rigidity of Riemannian structures, for example the dynamical behaviour of the Ricci flow.\\
 In this paper, we extend the definition of the Einstein tensor, where we calculate the first variation of the $F$-Einstein-Hilbert functional, and we conclude the generalized Einstein tensor.
 We prove that the generalized Einstein tensor is divergence-free. We study the second variation of the $F$-Einstein-Hilbert functional on the Riemannian manifold. The second variation formula
  gives a tool/is a prerequisite for the study the stability of any generalized Einstein manifold, and to see if the $F$-Einstein-Hilbert functional has extremality properties at some critical points.
  The smooth function $F$ can be chosen for the existence and the stability of such Riemannian metrics which provide additional information on Riemannian manifolds.

\section{$F$-Einstein-Hilbert functional}
First, we give some definitions. Let $(M,g)$ be an $n$-dimensional Riemannian manifold, and let $X,X_1,...,X_{q-1},Y,Z\in \Gamma(TM)$.
By $R$, $\operatorname{Ric}$ and $S$
we denote respectively the Riemannian curvature tensor, the Ricci tensor and the scalar curvature of  $(M,g)$.
Thus $R$, $\operatorname{Ric}$ and $S$ are defined by
\begin{equation}\label{curvature}
    R(X,Y)Z=\nabla_X \nabla_Y Z-\nabla_Y \nabla_X Z-\nabla_{[X,Y]}Z,
\end{equation}
\begin{equation}\label{ricci}
    \operatorname{Ric}(X,Y)=g(R(X,e_i)e_i,Y),\quad S=\operatorname{Ric}(e_i,e_i),
\end{equation}
where $\nabla$ is the Levi-Civita connection with respect to $g$, $\{e_1,...,e_n\}$ is an orthonormal frame.\\
Given a smooth function $f$ on $M$,
 the gradient of $f$ is defined by
\begin{equation}\label{gradient}
    g(\operatorname{grad}f,X)=X(f),
\end{equation}
     the Hessian of $f$ is defined by
\begin{equation}\label{hessian}
    (\operatorname{Hess} f)(X,Y)=g(\nabla_X \operatorname{grad}f,Y),
\end{equation}
 the Laplacian of $f$ is defined by
\begin{equation}\label{laplacian}
 \Delta f=-\operatorname{Tr} (\operatorname{Hess} f).
\end{equation}
The divergence of $(0,q)$-tensor $\alpha$ on $M$ is defined by
\begin{equation}\label{divergence}
    (\delta \alpha)(X_1,...,X_{q-1})=-(\nabla_{e_i}\alpha)(e_i,X_1,...,X_{q-1}).
\end{equation}
The formal adjoint of the divergence $\delta:\Gamma(\otimes^{2}T^*M)\longrightarrow\Gamma(T^*M)$ is the map $\delta^*:\Gamma(T^*M)\longrightarrow\Gamma(\otimes^{2}T^*M)$ defined by
\begin{equation}\label{adjoint of delta}
    (\delta^* \alpha)(X,Y)=\frac{1}{2}\big((\nabla_X \alpha)Y+(\nabla_Y\alpha)X\big).
\end{equation}
The formal adjoint of the Levi-civita connection $\nabla$ is given by
\begin{equation}\label{adjoint of nabla}
(\nabla^*\alpha)(X_1,...,X_{q-1})=-(\nabla_{e_i}\alpha)(e_i,X_1,...,X_{q-1}),
\end{equation}
where $\alpha\in\Gamma(T^*M\otimes T^{(p,q)}M)$, and $\{e_1,...,e_n\}$ is an orthonormal frame.\\
The composition of $T,Q\in\Gamma(\odot^2 T^*M)$ is defined by
\begin{equation}\label{composition}
    (T\circ Q)(X,Y)=T(X,e_i)Q(Y,e_i),
\end{equation}
where $\{e_1,...,e_n\}$ is an orthonormal frame on $M$.\\
For $T\in\Gamma(\otimes^{2}T^*M)$, we define $T^\sigma\in\Gamma(\odot^2 T^*M)$ by
\begin{equation}\label{sigma}
T^\sigma(X,Y)=\frac{1}{2}\big(T(X,Y)+T(Y,X)\big).
\end{equation}
We define an endomorphism $\overset{\circ}{R}:\Gamma(\odot^2 T^*M)\longrightarrow\Gamma(\odot^2 T^*M)$ by
\begin{equation}\label{endomorphism}
    (\overset{\circ}{R} T)(X,Y)=T(R(e_i,X)Y,e_i).
\end{equation}
For $T\in\Gamma(\odot^2 T^*M)$, we define the Lichnerowicz Laplacian by
\begin{equation}\label{Lichnerowicz Laplacian}
    \Delta_L T=\nabla^*\nabla T+\operatorname{Ric}\circ T+T\circ \operatorname{Ric}-2\overset{\circ}{R} T.
\end{equation}
(For more details, see for example \cite{Besse}, \cite{ON}).

\begin{definition}[\cite{Felice}, \cite{Sotiriou}]\label{definition1}
We let $\mathcal{M}$ denote the space of Riemannian metrics on a closed orientable manifold $M$.
The generalized Einstein-Hilbert functional (or $F$-Einstein-Hilbert functional) is defined by
\begin{equation}\label{eq1}
\mathcal{E}_{F}:\mathcal{M}\longrightarrow\mathbb{R},\quad g\longmapsto \mathcal{E}_{F}(g)=\int_M F(S)v^g,
\end{equation}
where $S$ is the scalar curvature of $(M,g)$, and $F:\mathbb{R}\longrightarrow\mathbb{R}$ is a non-constant smooth
function.
\end{definition}

The Definition \ref{definition1}, is a natural generalization of Einstein-Hilbert functional or the total
scalar curvature, when $F$ is the identity map, then $\mathcal{E}_{F}$ reduces to the usual Einstein-Hilbert functional whose second order infinitesimal behaviour is well understood
(see \cite{berger,Besse,calabi1,calabi,hilbert,Koiso1,Koiso,Koiso3,Koiso4,Richard}).\\
Let $(M,g)$ be a closed orientable Riemannian manifold. Consider a smooth one-parameter
variation of the metric $g$, i.e., a smooth family of metrics $(g_t)$ with $-\epsilon<t<\epsilon$, such that $g_0=g$. Take local coordinates $(x^i)$ on $M$, and write the metric on $M$ in the usual way as $g_t=g_{i,j}(t,x)dx^i\otimes dx^j$.
Write $h=(\partial g_t/\partial t)_{t=0}$, then $h\in\Gamma(\odot^2T^*M)$ is a symmetric 2-covariant tensor field on $M$, we get the following.

\begin{theorem}[\cite{Felice}, \cite{Sotiriou}]\label{th1}
The first variation of the $F$-Einstein-Hilbert functional in the direction of $h$ is given by the formula
\begin{equation}\label{eq2}
    \frac{d}{dt}\mathcal{E}_{F}(g_t)\Big|_{t=0}=-\int_M \langle E_F(g),h\rangle v^g,
\end{equation}
where $\langle ,\rangle $ is the induced Riemannian metric on $\otimes^2T^*M$,
\begin{eqnarray}\label{eq3}
E_F(g)
   &=&F'(S)\operatorname{Ric}-\operatorname{Hess} F'(S)
   -\big(\Delta F'(S)+\frac{1}{2}F(S)\big)g,
\end{eqnarray}
and $F'$ is the derivative of the function $F$.
\end{theorem}

\begin{definition}
$E_F(g)$ is called the generalized Einstein tensor (or $F$-Einstein tensor).
\end{definition}

For the proof of Theorem \ref{th1}, we need the following lemma.

\begin{lemma}[\cite{Reto}, \cite{Peter}] \label{lemma1} Let $(M,g)$ be a Riemannian manifold. Then, the differential at $g$, in the direction of $h$, of the volume element and the scalar curvature are given by the following formulas
\begin{equation}\label{eq4}
    \frac{\partial v^{g_t}}{\partial t}\Big|_{t=0}=\frac{1}{2}(\operatorname{Tr} h) v^g=\frac{1}{2}\langle g,h\rangle v^g,
\end{equation}
\begin{equation}\label{eq5}
   \frac{\partial S_{t}}{\partial t}\Big|_{t=0}=\Delta(\operatorname{Tr} h)+\delta(\delta h)-\langle \operatorname{Ric},h\rangle .
\end{equation}
\end{lemma}

\begin{proof}[Proof of Theorem \ref{th1}] First note that
\begin{equation}\label{eq6}
\frac{d}{dt}\mathcal{E}_{F}(g_t)\Big|_{t=0}=\int_M \Big[\frac{\partial F(S_{t})}{\partial t}v^{g_t}+F(S_{t})\frac{\partial v^{g_t}}{\partial t}\Big]_{t=0},
\end{equation}
for all $t\in(-\epsilon,\epsilon)$, we have
$$\frac{\partial F(S_{t})}{\partial t}=F'(S_{t})\frac{\partial S_{t}}{\partial t},$$
by the Lemma \ref{lemma1}, we obtain
\begin{eqnarray}\label{eq7}
\frac{\partial F(S_{t})}{\partial t} \Big|_{t=0}
   &=&\nonumber F'(S)\Delta(\operatorname{Tr} h)+F'(S)\delta(\delta h) \\
   & & -F'(S)\langle \operatorname{Ric},h\rangle .
\end{eqnarray}
Calculating in a normal frame at $x \in M$ we have
\begin{eqnarray}\label{eq8}
F'(S)\Delta(\operatorname{Tr} h)
   &=&\nonumber -F'(S)e_i\big(e_i(\operatorname{Tr} h)\big) \\
   &=&\nonumber -e_i\big(F'(S)e_i(\operatorname{Tr} h)\big)+e_i(F'(S))e_i(\operatorname{Tr} h)\\
   &=&\nonumber -e_i\big(F'(S)e_i(\operatorname{Tr} h)\big)+e_i\big(e_i(F'(S))\operatorname{Tr} h\big)\\
   &&-e_i\big(e_i(F'(S))\big)\operatorname{Tr} h,
\end{eqnarray}
so, the first term in the right-hand side of (\ref{eq7}), is given by
\begin{eqnarray}\label{eq9}
F'(S)\Delta(\operatorname{Tr} h)
   &=&\nonumber\delta\big(F'(S)d(\operatorname{Tr} h)\big)-\delta\big((\operatorname{Tr} h) dF'(S)\big)\\
   && +\Delta(F'(S))\langle g,h\rangle .
\end{eqnarray}
If $f\in C^\infty (M)$ and $\alpha\in \Gamma(T^*M)$, then (see \cite{Peter}, \cite{ON})
\begin{equation}\label{eq10}
    \delta(f\alpha)=-\langle df,\alpha\rangle +f\delta \alpha,
\end{equation}
with $\langle df,\alpha\rangle =\alpha(\operatorname{grad} f)$. Applying this formula, gives
\begin{equation}\label{eq11}
    F'(S)\delta(\delta h)=\delta\big(F'(S)\delta h\big)+\langle dF'(S),\delta h\rangle ,
\end{equation}
by using the following formula (see \cite{Peter})
\begin{equation}\label{eq12}
    (\delta T)(Z)=\delta\big(T(\cdot,Z)\big)+\frac{1}{2}\big\langle T,\mathcal{L}_Z g\big\rangle ,
\end{equation}
where $\mathcal{L}_Z g$ is the Lie-derivative of $g $ along $Z\in\Gamma(TM)$ (see \cite{ON}), and
$T\in\Gamma(\odot^2T^*M)$, we get
\begin{eqnarray}\label{eq13}
\langle dF'(S),\delta h\rangle
   &=&\nonumber (\delta h)\big( \operatorname{grad} F'(S)\big)\\
   &=&\nonumber \delta\big(h(\cdot,\operatorname{grad} F'(S))\big)
   +\frac{1}{2}\big\langle h,\mathcal{L}_{\operatorname{grad} F'(S)} g\big\rangle \\
   &=& \delta\big(h(\cdot,\operatorname{grad} F'(S))\big)
   +\big\langle h, \operatorname{Hess} F'(S) \big\rangle ,
\end{eqnarray}
by equations (\ref{eq11}) and (\ref{eq13}), the second term on the left-hand side of (\ref{eq7}) is
\begin{eqnarray}\label{eq14}
F'(S)\delta(\delta h)
   &=&\nonumber \delta\big(F'(S)\delta h\big)
     +\delta\big(h(\cdot,\operatorname{grad} F'(S))\big)\\
   &&+\big\langle h, \operatorname{Hess} F'(S) \big\rangle .
\end{eqnarray}
Substituting (\ref{eq9}) and (\ref{eq14}) in (\ref{eq7}), we obtain
\begin{eqnarray}\label{eq15}
\frac{\partial F(S_{t})}{\partial t} \Big|_{t=0}
   &=&\nonumber \delta\big(F'(S)d(\operatorname{Tr} h)\big)-\delta\big((\operatorname{Tr} h) dF'(S)\big)\\
   &&\nonumber +\Delta(F'(S))\langle g,h\rangle +\delta\big(F'(S)\delta h\big)\\
   &&\nonumber  +\delta\big(h(\cdot,\operatorname{grad} F'(S))\big)
   +\big\langle h, \operatorname{Hess} F'(S) \big\rangle \\
   &&-F'(S)\langle \operatorname{Ric},h\rangle .
\end{eqnarray}
From equation (\ref{eq15}) and the Lemma \ref{lemma1}, we have
\begin{eqnarray}\label{eq16}
\Big[\frac{\partial F(S_{t})}{\partial t}v^{g_t}+F(S_{t})\frac{\partial v^{g_t}}{\partial t}\Big]_{t=0}
&=&\nonumber \Big\{\delta\big(F'(S)d(\operatorname{Tr} h)\big)-\delta\big((\operatorname{Tr} h) dF'(S)\big)\\
   &&\nonumber +\Delta(F'(S))\langle g,h\rangle +\delta\big(F'(S)\delta h\big)\\
   &&\nonumber  +\delta\big(h(\cdot,\operatorname{grad} F'(S))\big)
   +\big\langle h, \operatorname{Hess} F'(S) \big\rangle \\
   &&\nonumber-F'(S)\langle \operatorname{Ric},h\rangle  \Big\}v^{g}+\frac{F(S)}{2}\langle g,h\rangle v^g.\\
\end{eqnarray}
Substituting the formula (\ref{eq16}) in (\ref{eq6}), and consider the divergence theorem (see \cite{baird}),
the Theorem \ref{th1} follows.
\end{proof}

\begin{remark}
Let $X,Y\in\Gamma(TM)$, we have
\begin{eqnarray*}
\operatorname{Hess} F'(S)(X,Y)
   &=& X(Y(F'(S)))-(\nabla_X Y)(F'(S)) \\
   &=& X(F''(S)Y(S))-F''(S)(\nabla_X Y)(S)\\
   &=& X(F''(S))Y(S)+F''(S)X(Y(S))-F''(S)(\nabla_X Y)(S)\\
   &=& F'''(S)X(S)Y(S)+F''(S)(\operatorname{Hess} S)(X,Y).
\end{eqnarray*}
According to this formula,
the $F$-Einstein tensor is given by
\begin{eqnarray}\label{eq20.1}
E_F(g)&=&\nonumber F'(S)\operatorname{Ric}-F''(S)\operatorname{Hess} S-F'''(S)dS\otimes dS\\
&&-\big(F''(S)\Delta S+F'''(S)|\operatorname{grad} S|^2+\frac{1}{2}F(S)\big)g.
\end{eqnarray}
\end{remark}

\begin{remark}
Let $(M,g)$ be a Riemannian manifold, we get the following
\begin{itemize}
  \item If $F(s)=s$, for all $s\in \mathbb{R}$, the $F$-Einstein tensor is given by the formula (see \cite{Besse}, \cite{Reto})
  \begin{eqnarray}\label{eq17}
E_F(g)=E(g)
   &=&\operatorname{Ric}
   -\frac{S}{2}g,
\end{eqnarray}
  is the Einstein tensor.
  \item If $F(s)=s^2$, for all $s\in \mathbb{R}$, the $F$-Einstein tensor is given by (see \cite{Besse}, \cite{calabi1}, \cite{calabi})
  \begin{eqnarray}\label{eq18}
E_F(g)
   &=&2S\operatorname{Ric}-2\operatorname{Hess} S
   -\big(2\Delta S+\frac{S ^2}{2}\big)g.
\end{eqnarray}
\end{itemize}
\end{remark}

From Theorem \ref{th1}, we deduce.

\begin{theorem}[\cite{Felice}, \cite{Sotiriou}]\label{th2}
A Riemannian metric $g$ is a critical point of the $F$-Einstein-Hilbert functional if and only if
\begin{eqnarray}\label{eq20}
F'(S)\operatorname{Ric}-\operatorname{Hess} F'(S)
   -\big(\Delta F'(S)+\frac{1}{2}F(S)\big)g=0,
\end{eqnarray}
where $F:\mathbb{R}\longrightarrow\mathbb{R}$ is a non-constant smooth function.
\end{theorem}

By taking traces in (\ref{eq20}), we obtain
\begin{eqnarray}\label{trace critical}
SF'(S)+(1-n)\Delta F'(S)
   -\frac{n}{2}F(S)=0.
\end{eqnarray}

\begin{theorem}\label{th3}
Let $(M,g)$ be a Riemannian manifold. Then, the divergence of the generalized Einstein tensor is zero (that is, $\delta E_F(g)=0$).
\end{theorem}

\begin{proof}
Let $F:\mathbb{R}\longrightarrow\mathbb{R}$ be a smooth function, calculating in a normal frame $\{e_i\}$ at $x\in M$, with $X=e_j$, we have
\begin{equation}\label{eq22}
    \delta E_F(g) (X)
   = -(\nabla_{e_i}E_F(g))(e_i,X)
   = -e_i\big(E_F(g)(e_i,X)\big),
\end{equation}
by the definitions of generalized Einstein tensor, and the Hessian tensor, we get
\begin{eqnarray}\label{eq23}
    E_F(g)(e_i,X)&=&\nonumber F'(S)\operatorname{Ric}(e_i,X)-g(\nabla_{e_i}\operatorname{grad} F'(S),X)\\
   &&-\big(\Delta F'(S)+\frac{1}{2}F(S)\big)g(e_i,X),
\end{eqnarray}
substituting (\ref{eq23}) in (\ref{eq22}), and consider the definition of gradient operator, we obtain
\begin{eqnarray*}
\delta E_F(g) (X)
   &=&-\operatorname{Ric}(\operatorname{grad}F'(S),X)-F'(S)e_i\big(\operatorname{Ric}(e_i,X)\big)\\
   &&+g(\nabla_{e_i}\nabla_X\operatorname{grad}F'(S),e_i)
     +X(\Delta F'(S))+\frac{1}{2}X(F(S)),
\end{eqnarray*}
by the definitions of the divergence, and the curvature tensor, with $[e_i,X]=0$, we conclude that
\begin{eqnarray}\label{eq24}
\delta E_F(g) (X)
   &=&\nonumber-\operatorname{Ric}(\operatorname{grad}F'(S),X)+F'(S)(\delta \operatorname{Ric})(X)\\
   &&\nonumber+g(R(e_i,X)\operatorname{grad}F'(S),e_i)
   +g(\nabla_X\nabla_{e_i}\operatorname{grad}F'(S),e_i)\\
     &&+X(\Delta F'(S))+\frac{1}{2}X(F(S)),
\end{eqnarray}
note that
\begin{equation}\label{eq25}
   \operatorname{Ric}(\operatorname{grad}F'(S),X)=g(R(e_i,X)\operatorname{grad}F'(S),e_i),
\end{equation}
\begin{equation}\label{eq26}
  F'(S)(\delta \operatorname{Ric})(X)=-\frac{1}{2}X(F(S))=-\frac{1}{2} F'(S)X(S),
\end{equation}
\begin{equation}\label{eq27}
    g(\nabla_X\nabla_{e_i}\operatorname{grad}F'(S),e_i)=-X(\Delta F'(S)).
\end{equation}
Substituting the formulas (\ref{eq25}), (\ref{eq26}) and (\ref{eq27}) in (\ref{eq24}), the Theorem \ref{th3} follows.
\end{proof}

\begin{remark}\label{re2.4}\quad
\begin{itemize}
  \item If $E_F(g)=f g$ for some function $f$ on $M$, then $f$ is constant function on $M$ (because $\delta E_F(g)=0$).
  \item The condition $E_F(g)=\lambda g$ is equivalent to
  \begin{equation}\label{einstein metric}
   F'(S)\operatorname{Ric}-\operatorname{Hess} F'(S)=\mu g,
  \end{equation}
  for some function $\mu$ on $M$,
  it is also equivalent to
  \begin{eqnarray}\label{einstein metric 2}
 F'(S)\operatorname{Ric}-F''(S)\operatorname{Hess} S-F'''(S)dS\otimes dS=\mu g,
\end{eqnarray}
  (see equation (\ref{eq20.1})).
  \item If $F(s)=s$, for all $s\in\mathbb{R}$, then $E_F(g)=\lambda g$ if $(M,g)$ is Einstein manifold,
that is $\operatorname{Ric}=\mu g$ for some constant $\mu$ (see \cite{Besse}).
\end{itemize}
\end{remark}

\begin{example}
Let $M=(0,\infty)\times\mathbb{R}^3$ equipped with the Riemannian metric $g=dt^2+t^2(dx^2+dy^2+dz^2)$.
Let $F(s)=s^\alpha$ for some constant $\alpha$. Then, $E_F(g)=0$ if and only if $\alpha=\frac{1\pm\sqrt{3}}{2}$.
\end{example}

\begin{example}
Let $M=\mathbb{S}^{n}\subset\mathbb{R}^{n+1}$ and $F:\mathbb{R}\longrightarrow\mathbb{R}$ a non-constant smooth function.
Then, the induced Riemannian metric $g^{\mathbb{S}^{n}}$ is a critical point of the $F$-Einstein-Hilbert functional
if and only if $F(s_0)=0$ and $F'(s_0)=0$ where $s_0=n(n-1)$ is the scalar curvature of $(\mathbb{S}^{n},g^{\mathbb{S}^{n}})$.
\end{example}

\begin{remark}
The previous examples prove the following results; There is no equivalence between $E_F(g)=0$ and $E(g)=0$ where $F$
is a non-constant smooth function. There exist Riemannian Einstein metrics which are critical points of the $F$-Einstein-Hilbert functional where $F(s)\neq s$.
\end{remark}

\section{The second variation of $\mathcal{E}_{F}$}
Let $M$ be a closed orientable manifold. We denote by
$$\mathcal{M}_{c}=\{g\in \mathcal{M}\,|\, \operatorname{Vol}(M,g)=\int_M v^g=c\},$$
for some constant $c> 0$. This is a submanifold of $\mathcal{M}$ of codimension $1$, and its tangent space at $g\in\mathcal{M}_{c}$ is given by
$$T_g\mathcal{M}_{c}=\{T\in \Gamma(\odot^2 T^*M)\,|\,\int_M\langle g,T\rangle v^g=0\}.$$
 A Riemannian metric $g$ is a critical point of $\mathcal{E}_{F}|_{\mathcal{M}_{c}}$ if and only if $E_F(g)$ is orthogonal to $T_g\mathcal{M}_{c}$, that is $E_F(g)=\lambda g$ for some constant $\lambda$.
In the following Theorem, we calculate the second derivative of $\mathcal{E}_{F}(g_t)$ at $t=0$ where $(g_t)$ ($-\epsilon<t<\epsilon$) is a smooth one-parameter variation of such Riemannian metric $g$ which enables us to know the extremality properties of $\mathcal{E}_{F}$. Write
\begin{equation}\label{eq32}
    h=\frac{\partial g_t}{\partial t}\Big|_{t=0},\quad k=\frac{\partial^2 g_t}{\partial t^2}\Big|_{t=0},
\end{equation}
then $h,k\in\Gamma(\odot^2T^*M)$. Under the notation above we have the following.

\begin{theorem}\label{th2.3}
Let $(M, g)$ be a closed orientable Riemannian manifold with volume $c$. Suppose that $E_F(g)=\lambda g$, for some constant $\lambda$, then the second variation of $\mathcal{E}_{F}|_{\mathcal{M}_{c}}$ at $g$ in the direction of $h$ is given by
\begin{equation*}
    \frac{d^2}{dt^2}\mathcal{E}_{F}(g_t)\Big|_{t=0}=\int_M \big\langle T_0(h)+T_1(h),h\big\rangle v^{g},
\end{equation*}
where $T_0(h)$, $T_1(h)$ are defined by
\begin{eqnarray*}
T_0(h)
   &=& -\frac{F'(S)}{2}\nabla^*\nabla h
       +F'(S)\overset{\circ}{R}h
       +F'(S)\delta^*(\delta h)
       +\frac{1}{2}F'(S)\operatorname{Hess}(\operatorname{Tr}h)\\
   & & +\frac{F'(S)}{2}\big[\Delta(\operatorname{Tr} h)
       +\delta(\delta h)\big]g
       -\frac{1}{2}\big[\lambda
       +\frac{1}{2}F(S)\big](\operatorname{Tr} h)g,
\end{eqnarray*}
\begin{eqnarray*}
  T_1(h)
   &=& -f\operatorname{Ric}+\operatorname{Hess}f+(\Delta f) g
     - h(\nabla_\cdot \operatorname{grad}F'(S),\cdot)^\sigma
    -(\nabla_\cdot h)(\cdot,\operatorname{grad}F'(S))^\sigma\\
    &&+\frac{1}{2}\nabla_{\operatorname{grad}F'(S)}h
    -\langle \delta h+\frac{1}{2}d(\operatorname{Tr} h),dF'(S)\rangle g
    -\frac{1}{2}(\Delta F'(S))(\operatorname{Tr} h)g\\
    &&+\frac{1}{2}\langle \operatorname{Hess}F'(S),h\rangle g,
\end{eqnarray*}
and $f=F''(S)\big[\Delta(\operatorname{Tr} h)+\delta(\delta h)-\langle \operatorname{Ric},h\rangle \big]$.
\end{theorem}
For the proof of Theorem \ref{th2.3}, we need the following lemmas.

\begin{lemma}\label{lemma2}
Let $T_t,Q_t\in\Gamma(\odot^2 T^*M)$ all dependent of time $t\in(-\epsilon,\epsilon)$ with $T_0=T$ and $Q_0=Q$. Then
\begin{equation*}
\frac{\partial}{\partial t}\Big|_{t=0}\big\langle T_t,Q_t\big\rangle _t=\big\langle \frac{\partial T_t}{\partial t}\Big|_{t=0},Q\big\rangle
+\big\langle T,\frac{\partial Q_t}{\partial t}\Big|_{t=0}\big\rangle
-2\big\langle T,h\circ Q\big\rangle ,
\end{equation*}
where $\langle ,\rangle _t$ is the induced Riemannian metric (with respect to $g_t$) on $\otimes^2T^*M$.
\end{lemma}
\begin{proof}
We have
    $$\big\langle T_t,Q_t\big\rangle _t=T_t ^{ij} Q_t ^{ab} g_t ^{ia} g_t ^{jb},$$
so that
\begin{eqnarray*}
\frac{\partial}{\partial t}\Big|_{t=0}\big\langle T_t,Q_t\big\rangle _t
   &=& \frac{\partial T_t ^{ij}}{\partial t}\Big|_{t=0} Q ^{ab} g ^{ia} g ^{jb}
    +T ^{ij} \frac{\partial Q_t ^{ab}}{\partial t}\Big|_{t=0} g ^{ia} g ^{jb} \\
   & &  +T ^{ij} Q ^{ab} \frac{\partial g_t ^{ia}}{\partial t}\Big|_{t=0} g ^{jb}
    +T ^{ij} Q ^{ab} g ^{ia} \frac{\partial g_t ^{jb}}{\partial t}\Big|_{t=0},
\end{eqnarray*}
since $\frac{\partial g_t ^{ia}}{\partial t}\big|_{t=0}=-g^{iu}g^{av}h_{uv}$ and $\frac{\partial g_t ^{jb}}{\partial t}\big|_{t=0}=-g^{ju}g^{bv}h_{uv}$ (see \cite{Reto}), we get
\begin{eqnarray*}
\frac{\partial}{\partial t}\Big|_{t=0}\big\langle T_t,Q_t\big\rangle _t
   &=& \big\langle \frac{\partial T_t}{\partial t}\Big|_{t=0},Q\big\rangle
+\big\langle T,\frac{\partial Q_t}{\partial t}\Big|_{t=0}\big\rangle \\
   & &  -T ^{ij} Q ^{ab} g^{iu}g^{av}h_{uv} g ^{jb}
    -T ^{ij} Q ^{ab} g ^{ia} g^{ju}g^{bv}h_{uv},
\end{eqnarray*}
note that
\begin{eqnarray*}
-T ^{ij} Q ^{ab} g^{iu}g^{av}h_{uv} g ^{jb}
    -T ^{ij} Q ^{ab} g ^{ia} g^{ju}g^{bv}h_{uv}
   &=& -2T ^{ij} Q ^{ab} g^{iu}g^{av}h_{uv} g ^{jb},
\end{eqnarray*}
on the other hand
\begin{eqnarray*}
 -2\langle T,h\circ Q\rangle
   &=& -2T^{ij}(h\circ Q)^{ub}g^{iu}g^{jb} \\
   &=&  -2T ^{ij} g^{av}h_{uv} Q ^{ab} g^{iu}g ^{jb}.
\end{eqnarray*}
\end{proof}

\begin{lemma}\label{lemma3} Let $(f_t)$ $(-\epsilon<t<\epsilon)$ be a time dependent family of smooth functions on $M$ with $f_0=f$. Then,
the first variation of the Hessian and the Laplacian are given by
\begin{eqnarray*}
\frac{\partial \operatorname{Hess}_t f_t}{\partial t}\Big|_{t=0}
   &=&  \operatorname{Hess}\Big(\frac {\partial f_t}{\partial t}\Big|_{t=0}\Big)-(\nabla_{\cdot}h)(\cdot,\operatorname{grad} f)^\sigma
   +\frac{1}{2}\nabla_{\operatorname{grad} f} h,
\end{eqnarray*}
\begin{eqnarray*}
\frac{\partial \Delta_t f_t}{\partial t}\Big|_{t=0}
   &=&\nonumber \Delta\Big(\frac {\partial f_t}{\partial t}\Big|_{t=0}\Big)
   -\big\langle \delta h+\frac{1}{2}d(\operatorname{Tr}h),df\big\rangle +\big\langle \operatorname{Hess} f,h\big\rangle ,
\end{eqnarray*}
where $\operatorname{Hess}_t f_t$ (resp. $\Delta_t f_t$) is the Hessian (resp. Laplacian) of $f_t$ with respect to the metric $g_t$.
\end{lemma}

\begin{proof}
By the definition of Hessian (\ref{hessian}),  we obtain
\begin{eqnarray}\label{eq45}
\frac{\partial \operatorname{Hess}_t f_t}{\partial t}(X,Y)
   &=&\nonumber \frac{\partial}{\partial t}\Big[X(Y(f_t))-(\nabla^t _X Y)(f_t)\Big] \\
   &=& \nonumber \frac{\partial}{\partial t}\Big[X(Y(f_t))-g(\nabla^t _X Y,\operatorname{grad} f_t)\Big]\\
   &=& \nonumber X(Y(\frac{\partial f_t}{\partial t}))-g(\frac{\partial}{\partial t}\nabla^t _X Y,\operatorname{grad} f_t)
    - g(\nabla^t _X Y, \operatorname{grad} \frac{\partial f_t}{\partial t}),\\
\end{eqnarray}
where $\nabla^t$ is the Levi-Civita connection with respect to $g_t$. The first variation
of the Levi-Civita connection in the direction of $h$ is given by the formula
\begin{equation}\label{eq46}
    g(\frac{\partial}{\partial t}\nabla^t _X Y\Big|_{t=0},Z)=\frac{1}{2}\Big[(\nabla_X h)(Y,Z)+(\nabla_Y h)(X,Z)-(\nabla_Z h)(X,Y)\Big],
\end{equation}
(see \cite{Besse}). Here $X,Y,Z\in\Gamma(TM)$ (all independent of time $t$). We conclude that
\begin{eqnarray}\label{eq470}
\frac{\partial \operatorname{Hess}_t f_t}{\partial t}\Big|_{t=0}
   &=& \nonumber \operatorname{Hess}\Big(\frac {\partial f_t}{\partial t}\Big|_{t=0}\Big)-(\nabla_{\cdot}h)(\cdot,\operatorname{grad} f)^\sigma
   +\frac{1}{2}\nabla_{\operatorname{grad} f} h.\\
\end{eqnarray}
By the Lemma \ref{lemma2}, the first variation of $\Delta_t f_t$ is given by
\begin{eqnarray}\label{eq480}
\frac{\partial \Delta_t f_t}{\partial t}\Big|_{t=0}
   &=&\nonumber -\frac{\partial }{\partial t}\Big|_{t=0} \big\langle \operatorname{Hess}_t f_t,g_t\big\rangle _t\\
   &=&\nonumber  -\big\langle \frac{\partial }{\partial t}\operatorname{Hess}_t f_t\Big|_{t=0} ,g\big\rangle -\big\langle \operatorname{Hess} f,h\big\rangle
    +2\big\langle \operatorname{Hess} f,h\circ g\big\rangle ,\\
\end{eqnarray}
by equations (\ref{eq470}), (\ref{eq480}), with $h\circ g=h$, we have
\begin{eqnarray}\label{eq8}
\frac{\partial \Delta_t f_t}{\partial t}\Big|_{t=0}
   &=&\nonumber \Delta\Big(\frac {\partial f_t}{\partial t}\Big|_{t=0}\Big)+\operatorname{Tr} (\nabla_{\cdot}h)(\cdot,\operatorname{grad} f)^\sigma\\
   && -\frac{1}{2}\operatorname{Tr} \nabla_{\operatorname{grad} f} h+\big\langle \operatorname{Hess} f,h\big\rangle ,
\end{eqnarray}
and note that
\begin{equation}\label{eq01}
    \operatorname{Tr} (\nabla_{\cdot}h)(\cdot,\operatorname{grad} f)^\sigma
    =-\big\langle \delta h, df\big\rangle ,
\end{equation}
\begin{equation}\label{eq02}
    -\frac{1}{2}\operatorname{Tr} \nabla_{\operatorname{grad} f} h=
    -\frac{1}{2}\big\langle d(\operatorname{Tr}h),df\big\rangle .
\end{equation}
The proof is completed.
\end{proof}

\begin{proof}[Proof of Theorem \ref{th2.3}] First note that

\begin{equation}\label{eq33}
    \frac{d^2}{dt^2}\mathcal{E}_{F}(g_t)\Big|_{t=0}=-\frac{d}{dt}\Big|_{t=0}\int_M \langle E_F(g_t),\frac{\partial g_t}{\partial t}\rangle _{t}v^{g_t},
\end{equation}
by the variational formulas in Lemma \ref{lemma2}, we have
\begin{eqnarray}\label{eq34}
\frac{d^2}{dt^2}\mathcal{E}_{F}(g_t)\Big|_{t=0}
   &=&\nonumber -\int_M \langle \frac{\partial}{\partial t}E_F(g_t)\Big|_{t=0},h\rangle v^{g}\\
       &&\nonumber-\int_M \langle E_F(g),k\rangle v^{g} \\
   &&\nonumber  +2\int_M \langle E_F(g),h\circ h\rangle v^{g}\\
   &&-\frac{1}{2}\int_M\langle E_F(g),h\rangle (\operatorname{\operatorname{Tr}} h)v^g.
\end{eqnarray}
Since $E_F(g)=\lambda g$, we obtain
\begin{equation}\label{eq35}
   2\int_M \langle E_F(g),h\circ h\rangle v^{g}=2\lambda  \int_M |h| ^2v^{g},
\end{equation}
\begin{equation}\label{eq36}
  -\frac{1}{2}\int_M\langle E_F(g),h\rangle (\operatorname{Tr} h)v^g=  -\frac{\lambda}{2} \int_M(\operatorname{Tr} h)^2v^g.
\end{equation}
Since $\operatorname{Vol}(M,g_t)=c$ , we have
\begin{eqnarray}\label{eq37}
  \frac{d^2}{dt^2}\Big|_{t=0} \operatorname{Vol}(M,g_t)
    = \int_M  \frac{\partial^2 v^{g_t}}{\partial t^2}\Big|_{t=0}=0,
\end{eqnarray}
by equation (\ref{eq37}), and Lemma \ref{lemma1}, we get
\begin{eqnarray}\label{eq38}
  \frac{1}{2}\int_M  \frac{\partial}{\partial t}\Big|_{t=0} \Big[(\operatorname{Tr}_{t} \frac{\partial g_t}{\partial t})v^{g_t}\Big]=0,
\end{eqnarray}
where $\operatorname{Tr}_{t} \frac{\partial g_t}{\partial t}$ is the trace of $\frac{\partial g_t}{\partial t}$ with respect to $g_t$,
from equation (\ref{eq38}), and the Lemmas \ref{lemma1} and \ref{lemma2}, we obtain
\begin{eqnarray}\label{eq39}
  0&=&\nonumber\int_M  \Big[\frac{\partial}{\partial t} (\operatorname{Tr}_{t} \frac{\partial g_t}{\partial t})\Big|_{t=0}v^{g}
  + (\operatorname{Tr} h)\frac{\partial v^{g_t}}{\partial t}\Big|_{t=0}\Big]\\
  &=&\nonumber\int_M  \Big[\frac{\partial}{\partial t} \langle g_t, \frac{\partial g_t}{\partial t}\rangle _{t}\Big|_{t=0}v^{g}
  + \frac{1}{2}(\operatorname{Tr} h)^2v^g\Big]\\
  &=&\nonumber\int_M \big[|h| ^2+(\operatorname{Tr}k)-2\langle g,h\circ h\rangle + \frac{1}{2}(\operatorname{Tr} h)^2\big]v^g\\
  &=&\int_M \big[-|h| ^2+(\operatorname{Tr}k)+ \frac{1}{2}(\operatorname{Tr} h)^2\big]v^g,
  \end{eqnarray}
by equation (\ref{eq39}), the second term on the left-hand side of (\ref{eq34}) is
\begin{eqnarray}\label{eq40}
 -\int_M \langle E_F(g),k\rangle v^{g}
   &=&\nonumber -\lambda \int_M (\operatorname{Tr}k)v^g \\
   &=&  \int_M \big[-\lambda|h| ^2+ \frac{\lambda}{2}(\operatorname{Tr} h)^2\big]v^g.
\end{eqnarray}
We compute
\begin{eqnarray}\label{eq41}
  \frac{\partial}{\partial t}E_F(g_t)\Big|_{t=0}
   &=&\nonumber \frac{\partial F'(S_t)}{\partial t}\Big|_{t=0} \operatorname{Ric}+ F'(S)\frac{\partial \operatorname{Ric}_t}{\partial t}\Big|_{t=0}   \\
   & & \nonumber -\frac{\partial \operatorname{Hess}_t F'(S_t)}{\partial t}\Big|_{t=0} -\frac{\partial \Delta_t F'(S_t)}{\partial t}\Big|_{t=0}g\\
   & & -\Delta F'(S)h-\frac{1}{2}\frac{\partial F(S_t)}{\partial t}\Big|_{t=0} g-\frac{1}{2}F(S)h,
\end{eqnarray}
note that, from the Lemma \ref{lemma1}, we have
\begin{equation}\label{eq42}
    \frac{\partial F(S_t)}{\partial t}\Big|_{t=0}
    =F'(S)\big[\Delta(\operatorname{Tr} h)+\delta(\delta h)-\langle \operatorname{Ric},h\rangle \big],
\end{equation}
\begin{equation}\label{eq43}
    \frac{\partial F'(S_t)}{\partial t}\Big|_{t=0}
    =F''(S)\big[\Delta(\operatorname{Tr} h)+\delta(\delta h)-\langle \operatorname{Ric},h\rangle \big],
\end{equation}
by  Lemma \ref{lemma2}, and the definition of Lichnerowicz Laplacian, we get
\begin{eqnarray}\label{eq44}
\frac{\partial \operatorname{Ric}_t}{\partial t}\Big|_{t=0}
   &=&\nonumber \frac{1}{2 } \Delta_L h-\delta^*(\delta h)-\frac{1}{2}\operatorname{Hess}(\operatorname{Tr} h) \\
   &=&\nonumber \frac{1}{2} \nabla^*\nabla h-\overset{\circ}{R}h+\frac{1}{2}\operatorname{Ric}\circ h +\frac{1}{2}h\circ \operatorname{Ric} \\
   &&-\delta^*(\delta h) -\frac{1}{2}\operatorname{Hess}(\operatorname{Tr} h),
\end{eqnarray}
from equations (\ref{eq41}), (\ref{eq42}), (\ref{eq43}), (\ref{eq44}), and the Lemma \ref{lemma3}, we have
\begin{eqnarray}\label{eq45}
  \frac{\partial}{\partial t}E_F(g_t)\Big|_{t=0}
   &=&\nonumber f\operatorname{Ric}
    +F'(S)\Big[\frac{1}{2} \nabla^*\nabla h-\overset{\circ}{R}h+\frac{1}{2}\operatorname{Ric}\circ h +\frac{1}{2}h\circ \operatorname{Ric} \\
   &&\nonumber-\delta^*(\delta h) -\frac{1}{2}\operatorname{Hess}(\operatorname{Tr} h)\Big]
   -\operatorname{Hess}f
   +(\nabla_{\cdot}h)(\cdot,\operatorname{grad} F'(S))^\sigma\\
   &&\nonumber-\frac{1}{2}\nabla_{\operatorname{grad} F'(S)} h
   -(\Delta f)g
   +\big\langle \delta h+\frac{1}{2}d(\operatorname{Tr}h),dF'(S)\big\rangle g\\
   &&\nonumber-\big\langle \operatorname{Hess} F'(S),h\big\rangle g-(\Delta F'(S))h
   -\frac{F'(S)}{2}\Big[\Delta(\operatorname{Tr} h)\\
   &&+\delta(\delta h)-\langle \operatorname{Ric},h\rangle \Big] g
   -\frac{1}{2}F(S)h,
\end{eqnarray}
where $f=F''(S)\big[\Delta(\operatorname{Tr} h)+\delta(\delta h)-\langle \operatorname{Ric},h\rangle \big]$.\\
Note that, from the definitions of the composition (\ref{composition}) and the $F$-Einstein tensor (\ref{eq3}),
 and the condition $E_F(g)=\lambda g$, we get
\begin{eqnarray}\label{eq46}
  \frac{F'(S)}{2}\big(\operatorname{Ric}\circ h+h\circ\operatorname{Ric}\big)
   &=&\nonumber \big[\lambda+\Delta F'(S)+\frac{1}{2}F(S)\big]h  \\
   &&  +h(\nabla_{\cdot}\operatorname{grad}F'(S),\cdot)^\sigma,
\end{eqnarray}
\begin{eqnarray}\label{eq47}
\frac{F'(S)}{2}\langle \operatorname{Ric},h\rangle g
   &=&\nonumber\frac{1}{2} \big[\lambda+\Delta F'(S)+\frac{1}{2}F(S)\big] (\operatorname{Tr} h)g\\
   & &+\frac{1}{2}\langle \operatorname{Hess}F'(S),h\rangle g.
\end{eqnarray}
From equations (\ref{eq34}), (\ref{eq35}), (\ref{eq36}), (\ref{eq40}), (\ref{eq45}), (\ref{eq46}) and (\ref{eq47}), we have

\begin{eqnarray}\label{eq48}
\frac{d^2}{dt^2}\mathcal{E}_{F}(g_t)\Big|_{t=0}
   &=&\nonumber \int_M \Big\langle -f\operatorname{Ric}-\frac{F'(S)}{2}\nabla^*\nabla h+F'(S)\overset{\circ}{R}h\\
   &&\nonumber-h(\nabla_\cdot \operatorname{grad}F'(S),\cdot)^\sigma
               +F'(S)\delta^*(\delta h)\\
   &&\nonumber +\frac{1}{2}F'(S)\operatorname{Hess}(\operatorname{Tr}h)+\operatorname{Hess}f
               -(\nabla_\cdot h)(\cdot,\operatorname{grad}F'(S))^\sigma\\
   &&\nonumber+\frac{1}{2}\nabla_{\operatorname{grad}F'(S)}h
               +(\Delta f) g-\langle \delta h+\frac{1}{2}d(\operatorname{Tr} h),dF'(S)\rangle g\\
   &&\nonumber +\frac{F'(S)}{2}[\Delta(\operatorname{Tr} h)+\delta(\delta h)]g-\frac{1}{2}\big[\lambda
   +\Delta F'(S)\\
   &&+\frac{1}{2}F(S)\big](\operatorname{Tr} h)g+\frac{1}{2}\langle \operatorname{Hess}F'(S),h\rangle g,h\Big\rangle v^g,
\end{eqnarray}
the Theorem follows from equation (\ref{eq48}).
\end{proof}

\begin{remark}
If $F(s)=s$, for all $s\in\mathbb{R}$. Note that, the condition $E_F(g)=\lambda g$ is equivalent to $\operatorname{Ric}=\big[\lambda+\frac{S}{2}\big]g$. That is, $g$ is Einstein Riemannian metric with constant $\mu=\lambda+\frac{S}{2}$. In this case, we have
\begin{eqnarray*}
T_0(h)
   &=& -\frac{1}{2}\nabla^*\nabla h
       +\overset{\circ}{R}h
       +\delta^*(\delta h)
       +\frac{1}{2}\operatorname{Hess}(\operatorname{Tr}h)\\
   & & +\frac{1}{2}\big[\Delta(\operatorname{Tr} h)
       +\delta(\delta h)\big]g
       -\frac{\mu}{2}(\operatorname{Tr} h)g,
\end{eqnarray*}
and $T_1(h)=0$. From the formula
\begin{eqnarray*}
  (\operatorname{Tr} h)\delta(\delta h)
   &=& \delta\big((\operatorname{Tr} h)\delta h\big)+\delta\big(h(\cdot,\operatorname{grad}(\operatorname{Tr} h))\big)
    +\big\langle \operatorname{Hess}(\operatorname{Tr} h),h\big\rangle ,
\end{eqnarray*}
and the divergence theorem (see \cite{baird}), the second
variation of $\mathcal{E}_{F}|_{\mathcal{M}_{c}}$ at $g$ in the direction of $h$ is given by (see \cite{Besse}, \cite{Klaus})
\begin{eqnarray*}
\frac{d^2}{dt^2}\mathcal{E}_{F}(g_t)\Big|_{t=0}
   &=& \int_M \big\langle -\frac{1}{2}\nabla^*\nabla h
       +\overset{\circ}{R}h
       +\delta^*(\delta h)
       \\
   & & +\frac{1}{2}\Delta(\operatorname{Tr} h)g
       +\delta(\delta h)g
       -\frac{\mu}{2}(\operatorname{Tr} h)g,h\big\rangle v^{g}.
\end{eqnarray*}
\end{remark}

\begin{definition}\label{stable}
A Riemannian manifold $(M,g)$ is said to be $F$-Einstein if $E_F(g)=\lambda g$ for some constant $\lambda$, where $F:\mathbb{R}\longrightarrow\mathbb{R}$ is a non-constant smooth function. We call $\lambda$ the $F$-Einstein constant of $g$. We say that a closed orientable $F$-Einstein manifold is stable (resp. strictly stable) if for any $h\in TT=\operatorname{Tr}^{-1}(0)\cap\delta^{-1}(0)$ (such tensors are called transverse traceless or $TT$-tensors)
\begin{equation*}
    \mathcal{E}''_{F}(h)=\int_M \big\langle \widehat{T}_0(h)+\widehat{T}_1(h),h\big\rangle v^{g}\leq0\quad\hbox{(resp. $<0$)},
\end{equation*}
where $\widehat{T}_0$, $\widehat{T}_1$ are the restrictions of $T_0$, $T_1$ to $TT$ respectively, given by
\begin{eqnarray*}
\widehat{T}_0(h)
   &=& -\frac{F'(S)}{2}\big[\nabla^*\nabla h
       -2\overset{\circ}{R}h\big],
\end{eqnarray*}
\begin{eqnarray*}
  \widehat{T}_1(h)
   &=& -\widehat{f}\operatorname{Ric}+\frac{1}{2}\nabla_{\operatorname{grad}F'(S)}h,
\end{eqnarray*}
and $\widehat{f}=f|_{TT}=-F''(S)\langle \operatorname{Ric},h\rangle$.
\end{definition}

\begin{remark}\quad
\begin{itemize}
  \item In the Definition \ref{stable}, $\operatorname{Tr}^{-1}(0)$ (resp. $\delta^{-1}(0)$) denotes the space of  symmetric $(0,2)$-tensor fields, whose trace (resp. divergence) vanishes on $(M,g)$.
  \item By using $\delta h=0$ and symmetry of $h$, we obtain the following formulas
  $$\langle \operatorname{Hess}\widehat{f},h\rangle=-\delta\big[ h(\operatorname{grad}\widehat{f},\cdot)\big];$$
  $$- \langle h(\nabla_\cdot \operatorname{grad}F'(S),\cdot)^\sigma,h\rangle
    -\langle(\nabla_\cdot h)(\cdot,\operatorname{grad}F'(S))^\sigma,h\rangle$$
  \begin{eqnarray*}
           \quad\qquad\qquad\qquad\qquad\qquad&=& \delta\left[(h\circ h)(\operatorname{grad}F'(S),\cdot)\right].
  \end{eqnarray*}
This explains the disappearance of these terms in $\langle \widehat{T}_1(h),h\rangle$ after integration over $M$.
  \item The Definition \ref{stable}, is a natural generalization of stable Einstein manifold (see \cite{Besse,Klaus,Koiso,Koiso3}).
  \item We call the operator $\Delta_E^F(h)=-2\big(\widehat{T}_0(h)+\widehat{T}_1(h))$ the $F$-Einstein operator.
  Thus, an $F$-Einstein manifold $(M,g)$ is stable, if the $F$-Einstein operator is nonnegative on $TT$-tensors, and strictly stable if it is positive on $TT$-tensors. If $F(s)=s$ for all $s\in\mathbb{R}$, then the $F$-Einstein operator reduces to the usual Einstein operator $\Delta_E(h)=\nabla^*\nabla h
       -2\overset{\circ}{R}h$.
\end{itemize}
\end{remark}

\begin{theorem}
Let $F\in C^\infty(\mathbb{R})$. We assume that $F'(s)\geq0$ (resp. $F'(s)>0$) for all $s\in\mathbb{R}$. Then,
any Einstein manifold of negative sectional curvature is stable (resp. strictly stable) $F$-Einstein manifold.
\end{theorem}

\begin{proof}
  Let $(M,g)$ be an Einstein manifold with Einstein constant $\mu$, i.e., $\operatorname{Ric}=\mu g$. Thus,
  $(M,g)$ is $F$-Einstein manifold with $F$-Einstein constant $\lambda=\mu F'(S)-\frac{1}{2}F(S)$.
  Moreover, $\widehat{f}=0$ and the $F$-Einstein operator becomes
  $$\Delta_E^F(h)=F'(S)\big[\nabla^*\nabla h-2\overset{\circ}{R}h\big].$$ Hence, if $F'(S)\geq0$ (resp. $F'(S)>0$)
  and the sectional curvature of $(M,g)$ is negative, then $(M,g)$ is stable (resp. strictly stable) $F$-Einstein manifold  (see \cite{Klaus,Koiso1}).
\end{proof}

\begin{remark}
Let $(M,g)$ be an $F$-Einstein manifold with $F$-Einstein constant $\lambda$. We assume that $(M,g)$ has constant scalar curvature. If $F'(S)>0$, according to (\ref{eq3}), the Riemannian manifold $(M,g)$ is Einstein with Einstein constant $\mu=F'(S)^{-1}\left(\lambda+\frac{1}{2}F(S)\right)$. Moreover, if the sectional curvature of $(M,g)$ is negative, then $(M,g)$ is strictly stable. Here, if the manifold $M$ is even-dimensional, by using (\ref{trace critical}) with $E_F(g)=\lambda g$, we can consider the smooth function $F(s)=-2\lambda+c\,s^{n/2}$ for some $c\in\mathbb{R}$.
\end{remark}

\begin{theorem}\label{th20}
Let $F\in C^\infty(\mathbb{R})$ and $(M,g)$ be a closed orientable $F$-Einstein manifold of constant sectional curvature $c>0$. We assume that $F'(s)\geq0$ and $F''(s)\leq0$ for all $s\in\mathbb{R}$. Then, $(M,g)$ is stable.
Moreover, if $F'(s)>0$ for all $s\in\mathbb{R}$, then $(M,g)$ is strictly stable.
\end{theorem}

\begin{proof}
A straightforward calculation shows that if $h\in TT$,
\begin{eqnarray}\label{3.25}
    -\frac{F'(S)}{2}\langle\nabla^*\nabla h,h\rangle
    &=&\nonumber-\frac{1}{2}\delta\big[F'(S)\langle \nabla_\cdot h,h\rangle\big]
                -\frac{1}{2}\langle \nabla_{\operatorname{grad}F'(S)}h,h\rangle\\
    & &-\frac{F'(S)}{2}\operatorname{Tr}\langle\nabla_\cdot h,\nabla_\cdot h\rangle.
\end{eqnarray}
By using $\operatorname{Tr}h=0$, we find that
\begin{equation}\label{3.26}
    F'(S)\langle\overset{\circ}{R}h,h\rangle=-c\,F'(S)|h|^2.
\end{equation}
From equations (\ref{3.25}) and (\ref{3.26}), we conclude that
\begin{eqnarray}\label{3.26}
    \mathcal{E}''_{F}(h)
    &=&\nonumber\int_M \Big[-\frac{1}{2}F'(S)\operatorname{Tr}\langle\nabla_\cdot h,\nabla_\cdot h\rangle-c\,F'(S)|h|^2\\
    & &+F''(S)\langle \operatorname{Ric},h\rangle^2\Big] v^{g}.
\end{eqnarray}
Theorem \ref{th20} follows from equation (\ref{3.26}), the assumptions $F'\geq0$, $F''\leq0$, and $c>0$.
\end{proof}

\begin{corollary}
The $n$-dimensional unit sphere $\mathbb{S}^n$ is a strictly stable $F$-Einstein manifold for all
$F\in C^\infty(\mathbb{R})$ such that $F'>0$ and $F''\leq0$.
\end{corollary}


\subsection*{Conflict of interest statement}
The author declares no conflict of interest.

\subsection*{Data availability}
Not applicable.

\end{document}